\documentclass[10pt,oneside,leqno]{amsart}
\usepackage{amsxtra}
\usepackage{amsopn}
\usepackage{color}
\usepackage{amsmath,amsthm,amssymb}
\usepackage{amscd}
\usepackage{amsfonts}
\usepackage{latexsym}
\usepackage{verbatim}

\theoremstyle{plain}
\newtheorem{theorem}{Theorem}[section]
\newtheorem{definition}[theorem]{Definition}
\newtheorem{lemma}[theorem]{Lemma}
\newtheorem{proposition}[theorem]{Proposition}
\newtheorem{corollary}[theorem]{Corollary}
\newtheorem{remark}[theorem]{Remark}
\newtheorem{example}[theorem]{Example}
\newtheorem{question}[theorem]{QUESTION}
\newtheorem{remark-question}[section]{Remark-Question}
\newtheorem{conjecture}[section]{Conjecture}

\newcommand\R{{\mathbb R}}
\newcommand\Z{{\mathbb Z}}

\newcommand\SU{{\rm SU}}

\newcommand\frh{{\mathfrak h}}

\newcommand\frs{{\mathfrak s}}

\newcommand{\sla}{\slash\!\!\!}

\begin{document}
\title[]{On the geometry underlying supersymmetric flux vacua with intermediate SU(2) structure}


\author{Anna Fino}
\address[Fino]{Dipartimento di Matematica\\
 Universit\`a di Torino\\
Via Carlo Alberto 10\\
10123 Torino, Italy} \email{annamaria.fino@unito.it}

\author{Luis Ugarte}
\address[Ugarte]{Departamento de Matem\'aticas\,-\,I.U.M.A.\\
Universidad de Zaragoza\\
Campus Plaza San Francisco\\
50009 Zaragoza, Spain} \email{ugarte@unizar.es}

\maketitle

\begin{abstract}
We show that supersymmetric flux vacua with constant intermediate SU(2)
structure are closely related to some special classes of half-flat
structures. More concretely, solutions of the SUSY equations IIA
possess a symplectic half-flat structure, whereas solutions of the
SUSY equations IIB admit a half-flat structure which is in certain
sense near to the balanced condition. Using this result we show that
compact simply connected manifolds do not admit type IIB solutions.
New solutions of the SUSY equations IIA and IIB are constructed from
hyperk\"ahler 4-manifolds, special hypo 5-manifolds and
6-dimensional solvmanifolds.
\end{abstract}


\section{Introduction}
In \cite{And} new supersymmetric four-dimensional Minkowski flux
vacua of type II string theory  with at least $N = 1$ supersymmetry
(SUSY) on nilmanifolds and solvmanifolds have been found, by extending
previous results by \cite{GMPT}.

In  \cite{GMPT2,GMPT3} it was shown that these supersymmetric conditions can be written in terms of the so-called generalized complex geometry \cite{Gualtieri,Hitchin} and it was proved that the internal manifold has to be a (twisted) generalized Calabi-Yau  manifold.
An $N = 1$ supergravity vacuum  implies the existence of a pair of spinors on the internal manifold.
In dimension six  the pair of spinors defines an $SU(3)$ structure, a static $SU(2)$ structure  or an intermediate $SU(2)$ structure  \cite{And}, which correspond respectively to the condition that the two spinors are parallel, orthogonal or between the two. These different cases are encoded in the context of the generalized geometry into an $SU(3) \times SU(3)$ structure on the bundle $TM \oplus T^*M$. The $SU(3) \times SU(3)$ structure can be encoded in a pair of compatible pure spinors, which are objects defined in generalized complex geometry on the generalized
tangent bundle $TM \oplus T^* M$. When one of the two pure spinors is closed, the manifold is called generalized Calabi-Yau.

An interesting question is then to look for  new explicit examples and a natural class is the  one of the nilmanifolds, since by \cite{CG} they admit  generalized Calabi-Yau structures.
In \cite{GMPT} the authors only look for $SU(3)$ and static $SU(2)$ structures, since only these ones  seemed to be compatible with the orientifold projections. But in \cite{KT} it was shown that intermediate $SU(2)$ structures are also possible if one allows a mixing of the usual $SU(2)$ structure forms under the projection conditions.

 In \cite{And} Andriot rewrote the projection conditions imposed by the orientifold for intermediate $SU(2)$ structures by introducing the \lq \lq
projection (eigen)basis", i.e. the set of structure forms which are eigenvectors for the projection. These forms define a new $SU(2)$ structure,
obtained by a rotation from the usual one and the new $SU(2)$ structure coincides (modulo a rescaling) with the one appearing with dielectric pure spinors, introduced in \cite{HT,MPZ} in the ADS/CFT context.
Since the pure spinors become  simpler to study if they are written in terms of the projection basis variables,  the supersymmetry (SUSY) conditions become simple and in this way Andriot  found for constant intermediate $SU(2)$ structures new four dimensional (Minkowski) flux vacua of type II string theory with at least $N = 1$.

Inside the class of $SU(3)$ structures there is a special one which is strictly related to the construction of metrics with holonomy $G_2$.
An $SU(3)$ structure defines a non-degenerate 2-form $F$, an almost-complex structure $J$, and a complex volume form $\Psi$;
the $SU(3)$ structure is called {\em half-flat} if $F \wedge F$ and the real part of $\Psi$ are
closed \cite{ChSal}. Hypersurfaces in $7$-dimensional manifolds with holonomy $G_2$ have a natural
half-flat structure, given by the restriction of the holonomy group representation.
In \cite{H} Hitchin showed that, starting with a half-flat manifold $(M, F, \Psi)$, if certain evolution equations have a solution coinciding with $(F, \Psi)$
at time zero then $(M, F, \Psi)$ can be embedded isometrically as a hypersurface in a
manifold with holonomy contained in $G_2$.

If in addition $F$ is closed the half-flat structure is called symplectic, and a half-flat structure with closed complex volume form $\Psi$ is known as Hermitian balanced. Nilmanifolds of dimension~$6$ admitting invariant symplectic, resp. Hermitian balanced, half-flat structures have been classified in \cite{CT}, resp. \cite{U}.
Recently $6$-dimensional nilmanifolds carrying an invariant half-flat structure have been classified by Conti in \cite{C}, extending previous partial results \cite{CSw,CT,U,ChF}.

 In this paper we show that supersymmetric flux vacua with constant intermediate $SU(2)$
structure are closely related to some special classes of half-flat
structures. More concretely,  in Section  \ref{sectSUSY} we show  that solutions of the SUSY equations IIA
have a symplectic half-flat structure, whereas solutions of the
SUSY equations IIB admit a half-flat structure which is in certain
sense near to the Hermitian balanced condition. In particular, we prove that
compact simply connected manifolds do not admit type IIB solutions.
Solutions of the SUSY equations IIA and IIB are constructed from
hyperk\"ahler 4-manifolds and, more generally, from special hypo 5-manifolds, where by hypo we mean the
natural $SU(2)$ structure induced on hypersurfaces in $6$-dimensional manifolds with holonomy SU(3)
given by the restriction of the holonomy group representation \cite{CS}.
In the last section we consider 6-dimensional solvmanifolds having both symplectic and Hermitian balanced
half-flat structures, and using them we find new solutions of the SUSY equations IIA and IIB.
The nilmanifolds considered in this paper have appeared previously in \cite{GMPT}, where solutions with $SU(3)$
or static $SU(2)$ structure were found. On the other hand, on the solvmanifold of Example~\ref{solvable-1}
$SU(3)$ structure solutions were given in \cite{CFI,GMPT}
(see also \cite{AGMP}), however to our knowledge the solvmanifold of Example~\ref{solvable-2} has not appeared previously
in relation to the SUSY equations and provides a new class of solutions.
In Section~\ref{new-examples} it is also proved that in general the solutions of equations IIA or IIB
are not stable by small deformations inside the class of half-flat structures (see Proposition~\ref{non-stable}).

\section{Intermediate $SU(2)$ structures}

In this section we follow the conventions of \cite{And}, and recall the four-dimensional Minkowski flux
vacua conditions of type II string theory with at least $N = 1$ supersymmetry
as well as their relation to the structure group of the internal manifold.

As in \cite{And} we consider type II supergravity (SUGRA) backgrounds, which are
warped products of the Minkowski space $\R^{3,1}$ and of a
$6$-dimensional compact manifold $M^6$. These warped products have
a metric of the form
\begin{equation}\label{metric}
ds^2_{(10)} = e^{2 A (y)} \eta_{\mu \nu} dx^{\mu} dx^{\nu} + g_{\mu
\nu}(y) dy^{\mu} dy^{\nu},
\end{equation}
where $\eta$ is the diagonal Minkowski metric with signature
$(3,1)$. The solutions will also
have non zero background values for some of the RR and NS fluxes. Let $vol_{(4)}$ denote the warped $4$-dimensional
volume form. Poincar\'e invariance in
dimension~$4$ requires the fluxes living on Minkowski space to be proportional to $vol_{(4)}$, so
we will focus on non trivial fluxes living on the internal manifold $M^6$.

As in \cite{GMPT} the total internal RR field $F$ is given by
$$
\begin{array}{l}
{\rm IIA:}\quad F = F_0 + F_2 + F_4 + F_6,\\[4pt]
{\rm IIB:}\quad F = F_1 + F_3 + F_5,
\end{array}
$$
where $F_k$ is the internal $k$-form RR field, and  it is related
to the total $10$-dimensional RR field-strength $F^{(10)}$ by
$$
F^{(10)} = F + vol_{(4)} \wedge \lambda (* F).
$$
Here $*$ denotes the Hodge star operator on $(M^6, g)$ and $\lambda$ is given by
$$
\lambda (A_p) = (-1)^{\frac{p (p - 1)}{2}} A_p,
$$
for every $p$-form $A_p$.

In order to find such solutions one has to solve the equations of motion and the Bianchi identities for
the fluxes, however, for the class of supergravity backgrounds
we are interested in, the equations of motion for the metric and the dilaton $\phi$  are implied by the
Bianchi identities and the $10$-dimensional supersymmetry conditions \cite{KT}, so one can solve the latter.
These conditions are the annihilation of the supersymmetry variations
of the gravitino $\psi_{\mu}$ and the dilatino $\lambda$ given by (see \cite{GMPT2})
$$
\begin{array}{l}
\delta \psi_{\mu} = D_{\mu} \epsilon + \frac{1}{4} H_{\mu} {\mathcal P} \epsilon + \frac{1}{16} e^{\phi}
\sum_n  \sla \! {F_{2n}} \, \gamma_{\mu} {\mathcal P}_n \, \epsilon  \, ,\\[5pt]
\delta \lambda = \left(\sla{\partial} \phi + \frac{1}{2} \sla \! H {\mathcal P}\right) \epsilon
+ \frac{1}{8} e^{\phi}
\sum_n (-1)^{2n} (5-2n)\,  \sla \! {F_{2n}} \,
{\mathcal P}_n  \epsilon  \, ,
\end{array}
$$
with $H_{\mu} = \frac 12  H_{\mu \nu \rho} \gamma^{\nu \rho}$, $H$ being the NSNS flux.
For IIA, ${\mathcal P} = \gamma_{11}$ and ${\mathcal P}_n = \gamma_{11}^n \sigma^1$ for $n = 0,  \dots, 5$, while
for IIB, ${\mathcal P} = -\sigma^3$, ${\mathcal P}_n  = \sigma^1$ for $n = \frac32,\frac72$
and ${\mathcal P}_n =  i \sigma^2$ for $n = \frac12,\frac52,\frac92$.

The 10-dimensional supersymmetry parameter $\epsilon$ can be written as a pair $(\epsilon^1,\epsilon^2)$ of two Majorana-Weyl
supersymmetry parameters and, because of the product structure of the solution \eqref{metric}, there should exist independent globally
defined and non-vanishing spinors $\eta^j$ on $M^6$
such that each $\epsilon^j$ is given as
$$
\epsilon^j= \zeta^j \otimes \sum_a f_a^j \eta_a^j\  +\ c.c., \quad \quad j=1,2,
$$
where $\zeta^1$ and $\zeta^2$ are the 4-dimensional supersymmetry parameters.

In order to get (at least) $N=1$ supersymmetry, it is required the existence of (at least) a pair $(\eta^1,\eta^2)$
of globally defined and non-vanishing spinors on the internal manifold $M^6$ satisfying the SUSY conditions.
The existence of this pair of internal spinors generically implies that the structure group of the tangent bundle
over the internal manifold $M^6$ is reduced to a subgroup $G\subset SO(6)$. This is due to the fact that the spinors
which are globally defined must not transform under $G$ and therefore are singlets under the $SO(6)\rightarrow G$ decomposition.
The pair $(\eta^1,\eta^2)$ can be parametrized and different types of $G$-structures are defined on the internal manifold
depending on the values of the parameters. $SU(3)$ and intermediate $SU(2)$ structures on 6-manifolds arise naturally in this context
as we recall next.

The existence of a globally defined non-vanishing spinor $\eta_{+}$ on a $6$-dimensional manifold $M^6$
defines a reduction of the structure group of the tangent bundle over $M^6$ to $SU(3)$.
Therefore, on the internal manifold we have an almost Hermitian structure $(J, g)$ and a
$(3,0)$-form $\Psi$ such that
$$
F \wedge \Psi =0, \quad \frac{4}{3} F^3 = i \Psi \wedge \overline
\Psi \neq 0,
$$
where $F$ is the fundamental $2$-form associated to  $(J, g)$.
The spinor $\eta_+$ is a Weyl spinor and it is supposed to have positive chirality and unit norm. Complex conjugation acts sending $\eta_+$ in $\eta_-$.
The forms $(F, \Psi)$ can be obtained as bilinears
of the globally defined spinor. Indeed:
$$
F_{\mu \nu} = - i \eta_+^{\dagger} \gamma_{\mu \nu} \eta_+,
\quad\quad
\Psi_{\mu \nu \rho} = - i \eta_-^{\dagger} \gamma_{\mu \nu \rho} \eta_+.
$$
An $SU(2)$ structure on a $6$-dimensional manifold $M^6$
is defined by two orthogonal globally defined spinors $\eta_{+}$ and $\chi_{+}$,
which we can suppose of unit norm, or equivalently by an almost Hermitian structure $(J, g)$,
a $(1,0)$-form $\alpha$, a real $2$-form $\omega$ and a $(2,0)$-form $\Omega$
satisfying the following conditions
 $$
 \begin{array}{c}
\omega^2 =  \frac{1}{2} \Omega \wedge \overline \Omega \neq 0,\\[4pt]
\omega \wedge \Omega =0, \quad  \Omega \wedge \Omega =0,\\[4pt]
i_{\alpha} \Omega =0, \quad i_{\alpha} \omega =0,
 \end{array}
 $$
 where by $i_{\alpha}$ we denote the contraction by the vector field dual to $\alpha$ and we take $\alpha$ such that
$\| \alpha \|^2 = i_{\overline \alpha} \alpha =\overline{\alpha}_{\overline a} g^{{\overline a} b} \alpha _b =2$.

The forms $(\alpha, \omega, \Omega)$ are related to the globally defined spinors $(\eta_{+}, \chi_{+})$ by the relations
$$
\begin{array}{l}
\alpha_{\mu} = \eta_-^{\dagger} \gamma_{\mu} \chi_+,\\[4pt]
\omega_{\mu \nu} = -  i \eta_+^{\dagger}   \gamma_{\mu \nu} \eta_+ + i \chi_+^{\dagger} \gamma_{\mu \nu} \chi_+,\\[4pt]
\Omega_{\mu \nu}  =\eta_-^{\dagger} \gamma_{\mu \nu} \chi_-.
\end{array}
$$
The spinor $\chi_+$ can be rewritten in terms of $ \eta_-$ as
$\chi_+ = \frac 12 \alpha \eta_-$.

The $SU(2)$ structure is naturally embedded in the $SU(3)$ structure
defined by $\eta_+$ by:
\begin{equation}\label{asocSU(3)}
F = \omega + \frac{i}{2} \alpha \wedge \overline \alpha, \quad \
\Psi=\alpha \wedge \Omega.
\end{equation}
Conversely, if one has an $SU(3)$ structure $(F,\Psi)$ on $M^6$  and
a $(1,0)$-form $\alpha$ of norm~$\sqrt{2}$, then it has been proved  in \cite[Appendix A2]{And}  that $\omega$ and  $\Omega$ defined by these formulas
$$
\omega = F - \frac{i}{2} \alpha \wedge \overline \alpha, \quad\
\Omega = \frac{1}{2} i_{\overline \alpha}  \Psi
$$
provide an $SU(2)$ structure.

Given a pair $(\eta^1_+,\eta^2_+)$ of globally defined non-vanishing internal spinors corresponding to the internal components of the supersymmetry parameters, one can parametrize them as
\begin{equation}\label{supersympar}
\eta^1_+ = a \eta_+,\quad\quad
\eta^2_+ = b ( k_{||} \eta_+ + k_{\perp} \frac{1}{2} \alpha \eta_-),
\end{equation}
with $0 \leq k_{||} \leq 1$, $k_{\perp} = \sqrt{1 - k_{||}^2}$
and $a,b$ non-zero complex numbers such that $a= \| \eta^1_+ \| $ and $b = \| \eta^2_+ \|$.
As in \cite{And}, we consider $b=\bar{a}$ so that the relative phases of the spinors are fixed by $|a|$ and $\theta$, the latter given by $e^{i\theta}=\bar{a}/a$.

Now depending on the values of $k_{||} $ and $k_{\perp}$ one can define starting from the spinors different type of $G$ structures. Indeed, if $k_{\perp} =0$, or equivalently if $\eta^1_+$ and $\eta^2_+$ are parallel, then one has an $SU(3)$ structure. If $k_{\perp} \neq 0$ one has an $SU(2)$ structure and in the particular case when $k_{\perp} = 1$ and $k_{||} =0$ one gets the so called static $SU(2)$ structure. But one can consider as in \cite{And} the intermediate case $k_{\perp} \neq 0$ and $k_{||} \neq 0$. The two orthogonal spinors $\eta_+$ and $\chi_+ = \frac 12 \alpha \eta_-$ define an $SU(2)$ structure $(J,g,\alpha, \omega, \Omega)$ on the internal manifold $M^6$ as above and, relating the numbers $k_ {||}$ and $k_{\perp}$ to the angle $\phi \in [0, \frac{\pi}{2}]$ between the spinors by
$$
k_{||} = \cos (\phi), \quad k_{\perp} = \sin (\phi),
$$
one obtains (see \cite{Dall'Agata})
the family of $SU(3)$ structures on $M^6$ given by
$$
\begin{array}{l}
\tilde F_{\phi} = \cos (2 \phi) \omega + \frac{i}{2} \alpha \wedge \overline \alpha  + \sin(2 \phi) Re (\Omega),\\[4pt]
\tilde \Psi_{\phi}  =  \alpha \wedge (- \sin (2 \phi)  \omega
+ \cos (2 \phi) Re (\Omega) + i Im (\Omega)),
\end{array}
$$
or equivalently the family of $SU(2)$ structures
\begin{equation} \label{familysu(2)structures}
\begin{array}{l}
\tilde \omega_{\phi} = \cos (2 \phi)\, \omega +  \sin(2 \phi) Re (\Omega),\\[4pt]
\tilde \Omega_{\phi}  = - \sin (2 \phi)\, \omega +\cos (2 \phi) Re
(\Omega) + i Im (\Omega).
\end{array}
\end{equation}

\begin{definition}\cite{And} {\rm The $SU(2)$ structure on $(M^6, J, g, \alpha, \omega, \Omega)$ defined by
\eqref{familysu(2)structures} is called} intermediate
{\rm if
$k_{||}$ and $k_{\perp} $ are both different from zero. It is called {\em static} (or {\em orthogonal}) if $k_{\perp} =1$ and $k_{||}=0$.}
\end{definition}

We recall that an $SU(3)$ structure $(F,\Psi)$ is said to be {\em half-flat} if $d (F \wedge F)
=0$ and $d(Re(\Psi))=0$. If in addition $dF=0$, then the
$SU(3)$ structure is said to be {\em symplectic half-flat}.

\begin{definition}\label{HB}
{\rm An $SU(3)$ structure $(F,\Psi)$ on $M^6$ is called} Hermitian balanced {\rm if $d (F \wedge F) =0$
and $d\Psi =0$.}
\end{definition}

From now on by a {\em symplectic half-flat}, resp. {\em Hermitian balanced},
{\em $SU(2)$ structure} $(J, g, \alpha, \omega, \Omega)$ we mean that the associated $SU(3)$ structure given by \eqref{asocSU(3)} is symplectic
half-flat, resp. Hermitian balanced.

\section{SUSY equations} \label{sectSUSY}

In this section we show that supersymmetric flux vacua with intermediate $SU(2)$
structure are closely related to the existence of special classes of half-flat
structures on the internal manifold.
We begin by recalling the SUSY conditions derived by Andriot in \cite{And}.

To solve the SUSY conditions, rather than using Killing spinors methods or $G$-structures tools, in \cite{And} it was used
the formalism of generalized complex geometry \cite{Gualtieri,Hitchin,GMPT}.  In generalized complex  geometry  for a  $d$-dimensional manifold $M$, one considers the bundle $TM \oplus T^*M$, whose sections are
generalized vectors (sums of a vector and a $1$-form). The spinors on $TM \oplus T^*M$
are  Majorana-Weyl ${\mbox{Cliff}}(d, d)$ spinors, and locally they can be seen as polyforms, i.e. sums of even/odd
differential forms, which correspond to positive/negative chirality spinors. A ${\mbox{Cliff}}(d, d)$ spinor is pure
if it is annihilated by half of the ${\mbox{Cliff}}(d, d)$ gamma matrices. Such pure spinors can be obtained also  as
tensor products of ${\mbox{Cliff}}(d)$ spinors.

In the supergravity context,  the ${\mbox{Cliff}}(6, 6)$ pure spinors are defined as a biproduct
$\Phi_{\pm} = \eta^1_+ \otimes  {\eta_{\pm}^{2 \dagger}}$ of the internal supersymmetry parameters and,
via the Fierz identity, they can be seen as polyforms
$$
\Phi_{\pm}= \eta^1_+ \otimes  {\eta_{\pm}^{2 \dagger}} = \frac{1}{8} \sum_{k =0}^6 \frac{1}{k!}  ({\eta_{\pm}^{2 \dagger}} \gamma_{\mu_k \dots \mu_1} \eta_+^1) \gamma^{\mu_1 \dots \mu_k}.
$$
The explicit expressions of the two pure spinors  in terms of the forms $(\alpha, \omega, \Omega)$  are then
$$
\begin{array}{l}
\Phi_+ = \frac{\vert a \vert^2} {8} e^{- i \theta} e^{\frac{1}{2} \alpha \wedge \overline \alpha} (k_{||} e^{- i \omega} - i k_{\perp} \Omega),\\[4pt]
\Phi_- = - \frac{\vert a \vert^2} {8}  \alpha \wedge (k_{\perp} e^{- i \omega} + i k_{||} \Omega).
\end{array}
$$
In general, by \cite{Gualtieri} a pure spinor $\Phi$ can be written as
$$
\Phi= \Omega_k \wedge e^{B + i  K},
$$
where ${\Omega}_k$ is a holomorphic $k$-form, and $B$ and $K$ are real $2$-forms.
The rank $k$  of the form $\Omega_k$ is the
type of the spinor. For the intermediate $SU(2)$ structure where both $k_{||}$ and $k_{\perp}$ are different from zero, by  \cite{And}, the two pure spinors $\Phi_+$ and $\Phi_-$  can be rewritten as
$$
\Phi_+ = \frac{\vert a \vert^2} {8} e^{- i \theta} k_{||} e^{\frac{1}{2} \alpha \wedge \overline \alpha - i \omega - i  \frac{k_{\perp}}{k_{||}} \Omega},\quad\quad
\Phi_- = - \frac{\vert a \vert^2} {8}  k_{\perp}  \alpha \wedge e^{- i \omega + i \frac{k_{||}}{k_{\perp}} \Omega}
$$
and thus $\Phi_+$ and $\Phi_-$ have respectively type $0$ and $1$. In the case of the $SU(3)$ structure (limit $k_{\perp} =0$),
$$
\Phi_+ =  \frac{\vert a \vert^2} {8} e^{- i \theta} e^{- i F},\quad\quad
\Phi_- = - i \frac{\vert a \vert^2} {8}  \Psi
$$
and the two pure spinors are of type $0$ and $3$, respectively.
In the case of a static $SU(2)$ structure (the other limit $k_{||}=0$) one has
$$
\Phi_+ = -i  \frac{\vert a \vert^2} {8} e^{- i \theta} \Omega \wedge e^{\frac{1}{2} \alpha \wedge\overline \alpha},\quad\quad
\Phi_- = -  \frac{\vert a \vert^2} {8}  \alpha \wedge e^{- i \omega}
$$
and the two  pure spinors are of type $2$ and $1$, respectively.

Two pure spinors are said to be compatible if they have three common annihilators. A pair of compatible pure spinors defines an $SU(3) \times SU(3)$ structure on $TM \oplus T^* M$. Depending on the
relation between the spinors $\eta_+^{1,2}$, this translates on $TM$  into the $SU(3)$, static $SU(2)$ or intermediate
$SU(2)$ structures discussed above. So the formalism of generalized complex geometry  allows to give a unified characterization  of the topological properties a $N = 1$ vacuum has to satisfy: it must admit an $SU(3) \times  SU(3)$
structure on $TM \oplus  T^* M$. And so to satisfy this condition, one may  verify that our vacua admit a pair of compatible pure spinors.

An $N = 1$ vacuum  should satisfy the SUSY conditions, the equations of
motion and the Bianchi identities for the fluxes.
By \cite{GMPT2,GMPT3} the SUSY equations can be  then written in
terms of pure spinors by
$$
\begin{array}{l}
(d - H \wedge) (e^{2 A - \phi} \Phi_1) =0,\\[4pt]
(d - H \wedge) (e^{A - \phi}  Re(\Phi_2)) =0,\\[4pt]
(d - H \wedge) (e^{3A - \phi}  Im (\Phi_2)) = \frac{e^{4A}}{8} *
\lambda(F),
\end{array}
$$
with $\Phi_1 = \Phi_{\pm}$, $\Phi_2 = \Phi_{\mp}$ for IIA/IIB
(upper/lower),  following the conventions of \cite{GMPT}.  The first of these equations implies that one of the two pure spinors (the one with the same parity as the RR fields) must be twisted (because of the $- H \wedge$) conformally
closed. A manifold admitting a twisted closed pure spinor is a twisted Generalized Calabi-Yau
(see \cite{H,GMPT}) and one looks for vacua on such manifolds.

The equations of motion of the fluxes are
$$
(d + H \wedge) (e^{4A} * F) =0, \quad d(e^{4A - 2 \phi} * H) = \mp
e^{4A} \sum_{p} F_p \wedge * F_{p + 2},
$$
with the upper/lower sign for IIA/IIB.

The Bianchi identities (we assume no NS source) are
$$
(d - H \wedge)  F = \delta (source), \quad d H =0,
$$
where $\delta (source)$  is the charge density of the allowed sources: space-filling $D$-branes or orientifold
planes ($O$-planes). In compactification to $4$-dimensional Minkowski, the trace of the energy momentum tensor must be zero. Then
$O$-planes are needed since they are the only known sources with a negative charge, that can thus
cancel the flux contribution to this trace. As in \cite{And} the RR Bianchi identities are then assumed to be
$$(d - H \wedge) F = \sum_i Q^i V^i,
$$
where $Q^i$  is the source charge and $V^i$  is (up to a sign) its internal co-volume (the co-volume of the cycle wrapped by the source). The sign of the $Q^i$ indicates whether the source is a $D$-brane ($Q^i > 0$) or an $O$-plane ($Q^i < 0$).

For intermediate $SU(2)$ structures (for which $\frac{k_{\perp}}{k_{||}}$
is constant) in the large volume limit from the SUSY conditions one gets that
the H Bianchi identities is automatically
satisfied. Furthermore, for this class of compactifications, it was shown in \cite{GMPT} that the equations of
motion for
the RR fluxes are implied by the SUSY conditions. And it was shown in \cite{KT} that the equation of
motion of H
is implied by the SUSY conditions and the Bianchi identities. Then, in order to find a solution, having a
pair of compatible pure spinors on an twisted generalized Calabi-Yau manifold  with at least one $O$-plane,  as in \cite{And} one has  to verify that the SUSY conditions and the RR Bianchi identities are satisfied.

The presence of $O$-planes implies that the solution has to be invariant under the
action of the orientifold.  As shown in \cite{KT} the first step to derive the orientifold projection on the pure spinors is to compute
those for the internal SUSY parameters. This can be done starting from the projection on the $10$-dimensional
SUSY spinorial parameters $\epsilon^i$  and then reducing to the internal spinors $\eta_{\pm}^i$ . In our
conventions, we have
$$
\begin{array}{l}
O5:\ \sigma(\eta_{\pm}^1)  = \eta_{\pm}^2, \quad \sigma(\eta_{\pm}^2)  = \eta_{\pm}^1,\\[4pt]
O6:\ \sigma(\eta_{\pm}^1)  = \eta_{\mp}^2, \quad \sigma(\eta_{\pm}^2)  = \eta_{\mp}^1,
\end{array}
$$
where $\sigma$  is the target space reflection in the directions transverse to the $O$-plane. Using the expressions  \eqref{supersympar} for the internal spinors, one gets as in \cite{And}  the following projection conditions at the orientifold
plane:
$$
\begin{array}{l}
O5:\ e^{i \theta} = \pm 1, \quad \alpha \perp O5, \\[4pt]
O6:\ e^{i \theta} \, {\mbox{free}}, \quad {\mbox {Re}} (\alpha) \| O6, \quad {\mbox{Im}} (\alpha) \perp 06.
\end{array}
$$
The previous conditions can be expressed on $\alpha$ as
$$
O5:\ \sigma (\alpha) = - \alpha, \quad\ \quad  O6:\  \sigma (\alpha) = \bar{\alpha}.
$$
By \cite{KT} if the $G$-structures are constant (as the one which are considering), and if we work
on nil/solvmanifolds (which will be our case), these conditions are valid everywhere (not only at the
orientifold plane). Starting from the projections on the $\eta_{\pm}^i$, as in \cite{KT,And} one may derive the projections of the pure spinors $\Phi_{\pm}$ and from them those for the $SU(2)$ structure forms. In particular one has
$$
\begin{array}{l}
O5:\ \sigma(\omega) = (k_{||}^2 - k_{\perp}^2) \omega + 2 k_{||} k_{\perp} {\mbox{Re}} (\Omega),\quad
\sigma (\Omega) = - k_{||}^2 \Omega + k_{\perp}^2 \overline \Omega + 2 k_{||} k_{\perp} \omega,\\[5pt]
O6:\ \sigma(\omega) = (k_{\perp}^2 - k_{||}^2) \omega - 2 k_{||} k_{\perp} {\mbox{Re}} (\Omega),\quad
\sigma (\Omega) =  - k_{\perp}^2  \Omega + k_{||}^2 \overline \Omega
- 2 k_{||} k_{\perp} \omega.
\end{array}
$$
By introducing as in \cite{KT}
$$
\begin{array}{l}
O5:\  k_{||} = \cos (\phi),\quad k_{\perp} = \sin (\phi),\quad 0 \leq \phi
\leq \frac{\pi}{2},\\[4pt]
O6:\ k_{||} = \cos (\phi + \frac{\pi}{2}) = - \sin(\phi),\quad k_{\perp} = \sin (\phi +  \frac{\pi}{2}) = \cos (\phi),\quad
- \frac{\pi}{2} \leq  \phi \leq 0,
\end{array}
$$
one gets  in both cases the following formulas:
$$
\begin{array}{l}
\sigma (\omega) = \cos (2 \phi) \omega + \sin (2 \phi) Re (\Omega),\\[4pt]
\sigma(Re(\Omega)) = \sin(2 \phi) \omega - \cos(2 \phi) Re(\Omega),\\[4pt]
\sigma(Im(\Omega)) = - Im(\Omega).
\end{array}
$$
Since the previous projection conditions are not very
tractable,
in \cite{And} he worked in the projection (eigen)basis
\begin{equation} \label{projbasis}
\begin{array}{l}
\omega_{\|} = \frac 12 (\omega + \sigma (\omega)), \quad \omega_{\perp} = \frac 12 (\omega - \sigma (\omega)),\\[4pt]
Re (\Omega)_{||} = \frac{1}{2} (Re (\Omega) + \sigma (Re (\Omega))),\quad
Re (\Omega)_{\perp} = \frac{1}{2} (Re (\Omega) - \sigma (Re
(\Omega))),
\end{array}
\end{equation}
which can  then be expressed in terms of the original $SU(2)$-structure as
$$
\begin{array}{l}
\omega_{\|} = \frac 12 ( (1+ \cos(2 \phi)) \omega + \sin (2 \phi) Re(\Omega)),\\[4 pt]
 \omega_{\perp} = \frac 12 ((1 - \cos(2 \phi) \omega - \sin (2 \phi) Re(\Omega)),\\[4 pt]
Re( \Omega)_{\|} = \frac{1}{2} ((1 - \cos(2 \phi)) Re(\Omega) + \sin(2 \phi) \omega),\\[4 pt]
  Re(\Omega)_{\perp} = \frac{1}{2} ((1 +\cos(2 \phi)) Re(\Omega) - \sin(2 \phi) \omega).
 \end{array}
$$
As in \cite{And} one takes  $e^A = | a |^2 = 1$ and go to the large volume
limit, i.e. $A = 0$ and $e^{\phi} = g_s$ constant. This is indeed the regime in which one  will look for solutions. The only remaining freedom is
$\theta$ that we do not really need to fix. Moreover,
we choose to look only for intermediate $SU(2)$ structure, i.e. with $k_{||} \neq 0$ and $k_{\perp} \neq 0$  constant.
Taking the coefficients constant is important because it simplifies drastically the search for solutions
and the SUSY conditions are much simpler.
In fact, by using the projection (eigen)basis \eqref{projbasis} and the results in \cite{GMPT2,GMPT3}, together with
further simplications as explained in \cite{And}, one can rewrite the SUSY equations in the following form:

\medskip

\noindent $\bullet$ \emph{SUSY equations IIA:}
\begin{equation}
\label{SUSYeqIIA} \left \{ \begin{array}{l}
d (Re (\alpha)) = 0,\\[6pt]
k_{||} H = k_{\perp} d (Im (\Omega)),\\[6pt]
 d(Re(\Omega)_{\perp}) = k_{||} k_{\perp} Re (\alpha)\wedge d(Im
(\alpha)),\\[6 pt]
H\wedge Re (\alpha) = -\frac{k_{\perp}}{k_{||}} d(Im (\alpha)\wedge
Re (\Omega)_{||}),
\end{array} \right.
\end{equation}
together with $F_0, F_2$ and $F_4$ given by
$$
\begin{array}{l}
g_s * F_0 = \frac{1}{2} k_{\perp} d (Im (\alpha)) \wedge
(Im(\Omega))^2 + \frac{1}{k_{||}} H \wedge Re(\alpha) \wedge
Re(\Omega)_{||},\\[4pt]
g_s * F_2 = - k_{||} d (Im(\alpha)) \wedge Im (\Omega) +
\frac{1}{k_{||}} d (Re(\Omega)_{||}) \wedge Re (\alpha),\\[4pt]
g_s *F_4 = - k_{\perp} d (Im (\alpha)).
\end{array}
$$

\medskip

\noindent $\bullet$ \emph{SUSY equations IIB:}
\begin{equation}
\label{SUSYeqIIB} \left \{ \begin{array}{l}
d (Re (\alpha)) =0,\\[6 pt]
d (Im (\alpha)) =0,\\[6 pt]
k_{||} H = k_{\perp} d (Im(\Omega)),\\[6pt]
Re(\alpha) \wedge H = - \frac{k_{\perp}}{k_{||}} Im(\alpha) \wedge d (Re(\Omega)_{\perp}),\\[6pt]
Im(\alpha) \wedge H = \frac{k_{\perp}}{k_{||}} Re(\alpha) \wedge d (Re(\Omega)_{\perp}),\\[6pt]
Re (\alpha) \wedge Im (\alpha) \wedge d(Re(\Omega)_{||}) =  - H
\wedge Im (\Omega),
\end{array} \right .
\end{equation}
together with $F_1$ and $F_3$ given by
$$
\begin{array}{l}
k_{\perp} e^{i \theta} g_s * F_1 = H \wedge Re (\Omega)_{||},\\[4pt]
k_{\perp} e^{i \theta} g_s * F_3 = d (Re (\Omega)_{||}).
\end{array}
$$

Note that $\left (\frac{1} {\cos {\phi}} Re (\Omega)_{\perp}, \frac{1} {\sin
{\phi}} Re (\Omega)_{||}, Im (\Omega), \alpha \right )$ define a new
$SU(2)$ structure $(\alpha, \omega', \Omega')$ on $M$ with
$$
\omega' = \frac{1} {\sin {\phi}} Re (\Omega)_{||}, \quad
\Omega' =  \frac{1} {\cos {\phi}} Re (\Omega)_{\perp} + i\, Im
(\Omega),
$$
and then a new $SU(3)$ structure $(F',\Psi')$. Moreover, since we will consider only O6 planes in IIA
and O5 planes in IIB, by using a local adapted basis for this
$SU(2)$-structure (see \cite{CS}), one has that
$(\hat{\alpha}, \hat{\omega}, \hat{\Omega})$ given by
$$
\hat{\alpha} = Re(\alpha) + i\, k_{||} Im (\alpha),\quad
\hat{\omega} = \frac{1} {k_{\perp}} Re (\Omega)_{\perp}, \quad
\hat{\Omega} =  Im (\Omega) - i\, \frac{1} {k_{||}} Re (\Omega)_{||},
$$
is also an $SU(2)$ structure on $M$ in the IIA case, and
$$
\hat{\alpha}_1 = k_{||}  Re (\alpha) + i\, Im (\alpha),\quad
\hat{\omega} = \frac{1} {k_{\perp}} Re (\Omega)_{||}, \quad
\hat{\Omega} = \frac{1} {k_{||}} Re (\Omega)_{\perp} + i\, Im (\Omega),
$$
$$
\hat{\alpha}_2 = k_{||}  Im (\alpha) - i\, Re (\alpha),\quad
\hat{\omega} = \frac{1} {k_{\perp}} Re (\Omega)_{||}, \quad
\hat{\Omega} = \frac{1} {k_{||}} Re (\Omega)_{\perp} + i\, Im (\Omega),
$$
are $SU(2)$ structures on $M$ in the IIB case. We will use these structures
in the next theorems, and the corresponding $SU(3)$ structures will be denoted by $(\hat{F}, \hat{\Psi})$.
Notice that the almost complex structure $\hat{J}$ and the metric $\hat{g}$ change with respect to those
given by $(\alpha, \omega', \Omega')$.

\begin{theorem}\label{IIA} Let $(M^6, J, g, \alpha, \omega, \Omega)$ be a $6$-dimensional manifold endowed with
an $SU(2)$ structure such that the $2$-forms $Re(\Omega)_{||},
Re(\Omega)_{\perp}, Im(\Omega)$ satisfy the
equations~\eqref{SUSYeqIIA}, then $M^6$ admits a symplectic
half-flat $SU(2)$ structure $(\hat J, \hat g, \hat \alpha,
\hat \omega, \hat \Omega)$ with $d (Re (\hat \alpha)) =0$.
Conversely, if $M^6$  has a symplectic half-flat $SU(2)$ structure
$(\hat J, \hat  g,  \hat \alpha, \hat \omega, \hat \Omega)$ such
that $d(Re( \hat \alpha)) =0$, then the forms  $(Re (\Omega)_{||} ,
Re (\Omega)_{\perp}, Im (\Omega), \alpha)$ defined by
$$
 \frac{1} {k_{\perp}} Re (\Omega)_{\perp} =\hat \omega, \quad  Im (\Omega) - i \frac{1} {k_{||}} Re (\Omega)_{||} =  \hat
\Omega, \quad   Re(\alpha) + i k_{||} Im (\alpha) = \hat  \alpha
$$
are a solution of the equations \eqref{SUSYeqIIA}.
\end{theorem}

\begin{proof} For type IIA the $2$-forms
$\frac{1} {k_{\perp}} Re (\Omega)_{\perp}, \frac{1} {k_{||}} Re
(\Omega)_{||}, Im (\Omega)$ together with the complex 1-form
$Re(\alpha) + i k_{||} Im (\alpha)$ define a new $SU(2)$ structure
with
$$
\hat \omega = \frac{1} {k_{\perp}} Re (\Omega)_{\perp}, \quad
\hat \Omega = Im (\Omega)-i \frac{1} {k_{||}} Re (\Omega)_{||} ,
\quad \hat \alpha = Re(\alpha) + i k_{||} Im (\alpha)
$$
and then  a new $SU(3)$ structure $(\hat J, \hat F, \hat \Psi)$, with
$$
\hat F = \hat \omega + k_{||} Re(\alpha) \wedge Im(\alpha),
\quad\quad \hat \Psi = \hat \alpha \wedge \hat \Omega.
$$
By the
second equation
of \eqref{SUSYeqIIA} we have
$$
H = \frac{k_{\perp}}{k_{||}} d (Im(\Omega)).
$$
Then by the last equation of \eqref{SUSYeqIIA} we obtain
$$
d (Im(\Omega)) \wedge Re(\alpha) = - d (Im(\alpha) \wedge
Re(\Omega)_{||}),
$$
i.e. that the real part of  the form $\hat \Psi = \hat \alpha
\wedge \hat \Omega$ is closed.

Moreover, by
$$
 d(Re(\Omega)_{\perp}) = k_{||} k_{\perp} Re (\alpha)\wedge d(Im
(\alpha)),
$$
it follows that
$$
d(Re(\Omega)_{\perp} + k_{||} k_{\perp} Re(\alpha) \wedge
Im(\alpha))= k_{\perp} d(\hat F) =0,
$$
and so we have a symplectic half-flat $SU(2)$ structure on $M^6$.
Conversely, if $M^6$  has a symplectic half-flat $SU(2)$ structure
$(\hat J, \hat  g,  \hat \alpha, \hat \omega, \hat \Omega)$ such
that $d(Re( \hat \alpha)) =0$,  we have that  the fundamental form
$$
\hat F = \hat \omega + Re(\hat \alpha) \wedge Im (\hat \alpha) =
\frac{1} {k_{\perp}} Re (\Omega_{\perp}) + k_{||} Re(\alpha) \wedge
Im(\alpha)
$$
is closed and thus the equation
$$d(Re(\Omega)_{\perp}) = k_{||} k_{\perp} Re (\alpha)\wedge d(Im
(\alpha))
$$
 in  \eqref{SUSYeqIIA} holds. By the closedness  of  the real part of the $(3,0)$-form
 $
 \hat \Psi  = \hat \alpha  \wedge \hat \Omega
 $ we have that  also the last  equation in  \eqref{SUSYeqIIA} is satisfied for $H =  \frac{k_{\perp}}{k_{||}} d (Im (\Omega))$.
\end{proof}

\begin{theorem}\label{IIB}
Let $(M^6, J, g, \alpha, \omega, \Omega)$ be a $6$-dimensional
manifold endowed with an $SU(2)$ structure such that the $2$-forms
$Re(\Omega)_{||}, Re(\Omega)_{\perp}, Im(\Omega)$ satisfy the
equations \eqref{SUSYeqIIB}, then $M^6$ admits two half-flat
$SU(2)$ structures $(\hat J_1, \hat g_1, \hat \alpha_1,
\hat \omega, \hat \Omega)$ and $(\hat J_2, \hat g_2, \hat
\alpha_2, \hat \omega, \hat \Omega)$ such that $d (\hat
\alpha_1) =0$ and $\hat \alpha_2= k_{||}\,  Im(\hat \alpha_1)-i
\frac{Re(\hat \alpha_1)}{k_{||}}$.

Conversely, let $M^6$ be a 6-dimensional manifold endowed with a
half-flat $SU(2)$ structure $(\hat J_1, \hat g_1, \hat \alpha_1,
\hat \omega,  \hat \Omega)$ satisfying $d (\hat \alpha_1) =0$; if
$k_{||}\in (0,1)$ is such that the $SU(2)$ structure $(\hat J_2,
\hat g_2, \hat \alpha_2= k_{||}\, Im(\hat \alpha_1)-i \frac{Re(\hat
\alpha_1)}{k_{||}}, \hat \omega, \hat \Omega)$ is half-flat, then
the forms $(Re (\Omega)_{||}, Re (\Omega)_{\perp}, Im (\Omega),
\alpha)$ defined by
$$
\frac{1} {k_{\perp}} Re (\Omega)_{||} =\hat \omega, \quad
\frac{1}{k_{||}} Re (\Omega)_{\perp} + i Im (\Omega) = \hat \Omega,
\quad k_{||} Re(\alpha) + i Im (\alpha) = \hat \alpha_1,
$$ where $k_{\perp}=\sqrt{1-k_{||}^2}$, are a solution of the equations \eqref{SUSYeqIIB}.
\end{theorem}

\begin{proof} As we already remarked previously for type IIB the $2$-forms
$\frac{1}{k_{||}} Re (\Omega)_{\perp}, \frac{1} {k_{\perp}} Re
(\Omega)_{||}, Im (\Omega), \alpha$ define a new $SU(2)$ structure
with
$$
\hat \omega = \frac{1} {k_{\perp}} Re (\Omega)_{||}, \quad \hat
\Omega =  \frac{1} {k_{||}} Re (\Omega)_{\perp} + i Im (\Omega),
\quad \hat \alpha_1 = k_{||}  Re (\alpha) + i Im (\alpha)
$$
and then a new $SU(3)$ structure $(\hat J_1, \hat F, \hat
\Psi_1)$, with
$$
\hat F = \hat \omega + k_{||} Re(\alpha) \wedge Im(\alpha),
\quad \hat \Psi_1 = \hat \alpha_1 \wedge \hat \Omega.
$$
Suppose that the $2$-forms $Re(\Omega)_{||}, Re(\Omega)_{\perp},
Im(\Omega)$ satisfy the equations \eqref{SUSYeqIIB}, then
by the first five equations
we have
$$
d (\alpha \wedge (Re (\Omega)_{\perp} + i Im (\Omega)) =0.
$$
Then, $d(Re (\hat \alpha_1 \wedge \hat \Omega)) =0$.

Since $(\hat \alpha_1,\hat \omega,\hat \Omega)$ is an SU(2) structure, the condition
$$
Im(\Omega)^2= \frac{1}{k_{\perp}^2} Re(\Omega)_{||}^2
$$
is satisfied \cite{CS} and by the last equation of \eqref{SUSYeqIIB} we get
$$
 \begin{array}{lcl}
 Re (\alpha) \wedge Im (\alpha) \wedge d(Re(\Omega)_{||}) &=&  - \frac{k_{\perp}}{k_{||}} d(Im(\Omega))
\wedge Im (\Omega)\\[5 pt]
&=&  - \frac{1}{k_{\perp} k_{||}}  d(Re(\Omega)_{||}) \wedge
Re(\Omega)_{||}.
\end{array}
$$
Therefore
 $$
 d ((k_{\perp} k_{||} Re (\alpha) \wedge Im (\alpha) + Re(\Omega)_{||})^2) = 0,
 $$
i.e. $d(\hat F \wedge \hat F) =0$. Then we have a half-flat
$SU(2)$ structure.

Consider
$$\hat\alpha_2 =k_{||}  Im (\alpha) -i Re (\alpha)$$ and define
$\hat\Psi_2=\hat \alpha_2 \wedge \hat \Omega$. We have two
SU(3) structures $(\hat F, \hat \Psi_1)$ and $(\hat F,
\hat \Psi_2)$. Indeed, $$Re (\hat \alpha_1)\wedge Im
(\hat \alpha_1)= Re (\hat \alpha_2)\wedge Im (\hat \alpha_2)=
k_{||} Re(\alpha) \wedge Im(\alpha),$$  so $\hat F$ is the same in
both cases.

Now, equation
$$
Im(\alpha) \wedge H =
\frac{k_{\perp}}{k_{||}} Re(\alpha) \wedge d
(Re(\Omega)_{\perp})
$$
implies that $d(Re\, \hat \Psi_1)=0$, whereas equation
$$
Re(\alpha) \wedge H = -\frac{k_{\perp}}{k_{||}} Im(\alpha) \wedge d
(Re(\Omega)_{\perp})
$$
implies that $d(Re\, \hat \Psi_2)=0$. In conclusion we have that $(\hat F,
\hat \Psi_1)$ and $(\hat F, \hat \Psi_2)$ are half-flat.

To prove the converse, we first notice that the fundamental form
${\hat F}_1$ is given by ${\hat F}_1=\hat \omega + Re(\hat \alpha_1)
\wedge Im(\hat \alpha_1)= \frac{1} {k_{\perp}} Re (\Omega)_{||} +
k_{||} Re(\alpha) \wedge Im(\alpha)$, and therefore the closedness
of the 4-form ${\hat F}_1\wedge {\hat F}_1$ implies the last
equation of \eqref{SUSYeqIIB} for $H=\frac{k_{\perp}}{k_{||}}
d(Im(\Omega))$.

Let us consider the complex 3-form $\hat \Psi_j = \hat \alpha_j
\wedge \hat \Omega$, $j=1,2$. Since $Re(\hat \Psi_j)=Re(\hat
\alpha_j)\wedge Re(\hat \Omega)- Im(\hat \alpha_j)\wedge Im(\hat
\Omega)$, we get that
$$
d(Re(\hat \Psi_j))= \frac{k_{||}}{k_{\perp}} Im(\hat \alpha_j)\wedge
H - \frac{1}{k_{||}} Re(\hat \alpha_j) \wedge d(Re(\Omega)_{\perp}).
$$
For $j=1$ we get the equation $Im(\alpha)\wedge H =
\frac{k_{\perp}}{k_{||}} Re(\alpha) \wedge d(Re(\Omega)_{\perp})$,
and for $j=2$ we get $Re(\alpha)\wedge H = -
\frac{k_{\perp}}{k_{||}} Im(\alpha) \wedge d(Re(\Omega)_{\perp})$,
so the equations \eqref{SUSYeqIIB} are satisfied.
\end{proof}

\begin{remark}\label{family}
{\rm Notice that in the conditions of Theorem \ref{IIB} we can
define a 1-parametric family of half-flat $SU(2)$ structures connecting the
two given structures. In fact, by considering the usual rotation
$$\hat \alpha_t= k_{||} \sin t\, Re (\alpha)+ k_{||} \cos t\, Im
(\alpha) +i \sin t\, Im (\alpha)- i\cos t\, Re(\alpha),$$ we have
that $(\hat F, \hat \alpha_t \wedge \hat \Omega)$ is half-flat
for any value of $t$. Notice that the fundamental 2-form $\hat F$
does not depend on $t$ because $Re (\hat \alpha_t)\wedge Im
(\hat \alpha_t)= k_{||} Re(\alpha) \wedge Im(\alpha)$. The almost
complex structure $\hat J_t$ is given with respect to the basis $\{
Re(\alpha), Im(\alpha) \}$ by
$$
\hat J_t=\left( \begin{array}{rl}
\sin t\,\cos t(\frac{1}{k_{||}}-k_{||}) & k_{||} \sin^2 t+\frac{\cos^2 t}{k_{||}} \\
 -\frac{\sin^2 t}{k_{||}} - k_{||}\cos^2 t & -\sin t\,\cos
 t(\frac{1}{k_{||}}-k_{||})
\end{array}\right)
$$}
\end{remark}

Next, using the characterization given in Theorem \ref{IIB}, we show
that if there is solution of the SUSY equations IIB for any
$k_{||}\in (0,1)$, then the manifold must be Hermitian balanced.

\begin{proposition}\label{balanced-1}
Let $(J, g, \alpha, \omega, \Omega)$ be a half-flat
$SU(2)$ structure  on a $6$-manifold  $M^6$ such that $d (\alpha) =0$ and for each $\lambda\in
(0,1)$ the $SU(2)$ structure $(J_{\lambda}, g_{\lambda},
\beta_{\lambda}= \lambda\, Im(\alpha)-i \frac{Re(\alpha)}{\lambda},
\omega, \Omega)$ is half-flat. Then, $(J, g, \alpha, \omega,
\Omega)$ is Hermitian balanced.
\end{proposition}

\begin{proof}
Let us consider the SU(3) structure $(F,\Psi)$ given by $F=\omega+ \frac{i}{2}\, \alpha\wedge \overline \alpha$ and
$\Psi=\alpha \wedge \Omega$. According to Definition~\ref{HB} we have to prove that $Im(\Psi)$ is a closed form.
Let $(F,\Phi_{\lambda})$ be the $SU(3)$ structure associated to $(J_{\lambda}, g_{\lambda},
\beta_{\lambda}, \omega, \Omega)$. Then,
the real and imaginary parts of the
forms $\Psi$ and $\Phi_{\lambda}$ are given by
$$
\begin{array}{l}
Re(\Psi)= Re(\alpha) \wedge Re(\Omega) - Im(\alpha) \wedge Im(\Omega),\\[5pt]
Im(\Psi)= Re(\alpha) \wedge Im(\Omega) + Im(\alpha) \wedge Re(\Omega),\\[5pt]
Re(\Phi_{\lambda})=\frac{Re(\alpha)}{\lambda} \wedge Im(\Omega) + \lambda \, Im(\alpha) \wedge Re(\Omega),\\[5pt]
Im(\Phi_{\lambda})= - \frac{Re(\alpha)}{\lambda} \wedge Re(\Omega) +\lambda \, Im(\alpha) \wedge Im(\Omega).
\end{array}
$$
The limit of the 3-form $\Phi_{\lambda}$ exists
when $\lambda \rightarrow 1$ and equals $-i\Psi$. Since $Re(\Phi_{\lambda})$ is closed
for any $\lambda\in (0,1)$, we conclude that $Im(\Psi)$ is closed.
Therefore, the $SU(2)$ structure $(J, g, \alpha, \omega, \Omega)$ is
Hermitian balanced.
\end{proof}

Given a Hermitian balanced $SU(2)$ structure $(J, g, \alpha, \omega,
\Omega)$ such that  $d (\alpha) =0$ and $Re(\alpha) \wedge d (Re
(\Omega)) =0$  one can construct a solution  of the SUSY equations
IIB for any $k_{||}\in (0,1)$.

\begin{proposition}\label{balanced}
Let $(J, g, \alpha, \omega, \Omega)$ be a Hermitian balanced
$SU(2)$ structure  on a $6$-manifold $M^6$ such that $d (\alpha) =0$ and $Re(\alpha) \wedge d
(Re(\Omega)) =0$, then  for each $\lambda\in (0,1)$ the
$SU(2)$ structure $(J_{\lambda}, g_{\lambda}, \beta_{\lambda}=
\lambda\, Im(\alpha)-i \frac{Re(\alpha)}{\lambda}, \omega,  i
\Omega)$ is half-flat.
\end{proposition}

\begin{proof}
We have  $\Psi_{\lambda} =  ( \lambda\, Im(\alpha)-i
\frac{Re(\alpha)}{\lambda}) \wedge (i Re(\Omega) - Im (\Omega))$ and
thus
$$
\begin{array}{l}
Re(\Psi_{\lambda})=- \lambda \, Im(\alpha) \wedge Im(\Omega) + \frac{Re(\alpha)}{\lambda} \wedge Re(\Omega),\\[5pt]
Im(\Psi_{\lambda})=\lambda \, Im(\alpha) \wedge Re(\Omega) + \frac{Re(\alpha)}{\lambda} \wedge Im(\Omega).\\[5pt]
\end{array}
$$
By the assumptions on the $SU(2)$ structure $(J, g, \alpha, \omega,
\Omega)$ we get in particular that
$$
\begin{array}{l}
Re(\alpha) \wedge d (Re(\Omega)) = Im (\alpha) \wedge d
(Im(\Omega))=0
\end{array}
$$
and therefore that $d(Re(\Psi_{\lambda})) =0$.
\end{proof}

\begin{example}\label{h4}
{\rm
Let us consider a nilmanifold corresponding to the nilpotent Lie algebra
$\frh_4=(0,0,0,0,12,14+23)$, that is, there is a basis $\{e^1,\ldots,e^6\}$ satisfying
$$
de^1=de^2=de^3=de^4=0,\quad de^5=e^{12}\quad de^6=e^{14}+e^{23}.
$$
We consider the structure $(F,\Psi)$ given by the 2-form
$F=e^{13}+e^{24}-e^{56}$ and the (3,0)-form $\Psi=(e^{1}+i\, e^{3})(e^{2}+i\,
e^{4})(e^{6}+i\, e^{5})$. Although $\frh_4$ admits Hermitian
balanced structures, the previous structure is only half-flat. In
fact, $F^2$ and $Re(\Psi)$ are closed, but $d(Im(\Psi))=-e^{1234}$.

For the complex 3-form $\Phi_{\lambda}= (\lambda
e^{3}-i\, \frac{e^{1}}{\lambda})(e^{2}+i\, e^{4})(e^{6}+i\, e^{5})$,
a direct calculation shows that
$$
d(Re(\Phi_{\lambda}))= \frac{1}{\lambda} d(e^{125}+e^{146}) +
\lambda\, d(e^{326}-e^{345}) = \frac{2 \lambda^2-1}{\lambda}
e^{1234},
$$
which implies that $Re(\Phi_{\lambda})$ is closed only for
$\lambda=\pm\frac{1}{\sqrt{2}}$. Notice that $Im(\Phi_{\lambda})$ is
closed for any $\lambda$, so $(F,\Phi_{\pm\frac{1}{\sqrt{2}}})$ are
Hermitian balanced and $(F,\Phi_{\frac{1}{\sqrt{2}}})$ provides a
solution to equations \eqref{SUSYeqIIB}. In fact, by
Theorem~\ref{IIB} we have the following explicit solution $(\alpha,
Re (\Omega)_{||}, Re (\Omega)_{\perp}, Im (\Omega))$ of the SUSY
equations IIB for $k_{||}=k_{\perp}=\frac{1}{\sqrt{2}}$:
$$
Re(\alpha)= \sqrt{2}\, e^1,\quad\ Im(\alpha)=e^3,\quad\ Re
(\Omega)_{||}= \frac{1}{\sqrt{2}} (e^{24}-e^{56}),$$

$$Re (\Omega)_{\perp}= \frac{1}{\sqrt{2}} (e^{26}- e^{45}),\quad\quad Im
(\Omega)= e^{25}+ e^{46}.$$ Notice that the fluxes are:
$$H=-e^{234},\quad\ e^{i\theta}g_s*F_3= -e^{126}+e^{145}+e^{235},\quad\  e^{i\theta}g_s*F_1=
e^{23456}.$$

Now, let us start from the (Hermitian balanced) structure
$\Phi_{\lambda}$ for $\lambda=-\frac{1}{\sqrt{2}}$, that is,
$$
\Phi_{-\frac{1}{\sqrt{2}}}= (-\frac{e^{3}}{\sqrt{2}}+i\sqrt{2}\,
e^{1})(e^{2}+i\, e^{4})(e^{6}+i\, e^{5}),
$$
and consider the complex 3-form $\Theta_{\mu}$ given by
$$
\Theta_{\mu}=(\sqrt{2}\mu\, e^{1} + i
\frac{e^{3}}{\sqrt{2}\mu})(e^{2}+i\, e^{4})(e^{6}+i\, e^{5}).
$$
It is easy to check that $Re(\Theta_{\mu})$ is closed for any value
of $\mu$. Notice that in particular
$\Theta_{\frac{1}{\sqrt{2}}}=\Psi$, i.e. one member in the family
is precisely the half-flat structure $(F,\Psi)$ given at the beginning.
Again, by Theorem~\ref{IIB} we have the following solution $(\alpha,
Re (\Omega)_{||}, Re (\Omega)_{\perp}, Im (\Omega))$ of the SUSY
equations IIB for any $k_{||}=\mu\in (0,1)$:
$$
Re(\alpha)= -\frac{1}{\sqrt{2}\, k_{||}}\, e^3,\quad\ Im(\alpha)=
\sqrt{2}\, e^1,\quad\ Re (\Omega)_{||}= k_{\perp} (e^{24}-e^{56}),$$

$$Re (\Omega)_{\perp}= k_{||} (e^{26}- e^{45}),\quad\quad Im
(\Omega)= e^{25}+ e^{46},$$ where $k_{\perp}=\sqrt{1-k_{||}^2}$. The
fluxes are:
$$H= -\frac{k_{\perp}}{k_{||}} e^{234},\quad\ e^{i\theta}g_s*F_3= -e^{126}+e^{145}+e^{235},\quad\  e^{i\theta}g_s*F_1=
\frac{k_{\perp}}{k_{||}} e^{23456}.$$
}
\end{example}

From Theorem~\ref{IIA} it follows that a compact 6-manifold $M^6$
admitting a solution to equations \eqref{SUSYeqIIA} satisfies those
topological restrictions imposed by the existence of a symplectic
form, in particular the Betti numbers $b_2(M^6)$ and $b_4(M^6)$ do not vanish.
Next we prove that $b_1(M^6)\geq 2$ for any compact manifold $M^6$
admitting solution to \eqref{SUSYeqIIB}. In
Examples~\ref{solvable-1} and~\ref{solvable-2} we show that this
lower bound can be attained.

\begin{proposition} Let $(M^6, J, g, \alpha, \omega, \Omega)$ be a $6$-dimensional compact
manifold endowed with an $SU(2)$ structure such that the $2$-forms
$Re(\Omega)_{||}, Re(\Omega)_{\perp}, Im(\Omega)$ satisfy the
equations \eqref{SUSYeqIIB}, then $M^6$ has first Betti number $\geq
2$. In particular, there is no solution on compact simply connected
$6$-dimensional manifolds.
\end{proposition}

\begin{proof} From equations \eqref{SUSYeqIIB} we can prove that the closed
1-forms $Re(\alpha)$ and $Im(\alpha)$ are harmonic with respect to
$g$. In fact, the $5$-form $* Re(\alpha)$ is closed because it is a
(constant) multiple of $Im(\alpha) \wedge (Im(\Omega))^2$, which is
closed by the last equation of  \eqref{SUSYeqIIB}, taking into
account the value of $H$. Similarly,  the $5$-form $* Im(\alpha)$ is
also closed.
\end{proof}

Let us remind that a Riemannian manifold $(N,g)$ is called {\it
hyperk\"ahler} if there are three complex structures, I, J, K on $N$
satisfying the quaternion relations $$I^2=J^2=K^2=-1, \quad
IJ=K=-JI,$$ and such that $I,J,K$ are parallel. In particular, we
have three K\"ahler forms $\omega_I$, $\omega_J$ and $\omega_K$ on
$N$.

\begin{proposition}\label{HK}
Let $(N^4,I,J,K)$ be a compact 4-dimensional hyperk\"ahler manifold.
Then, on the 6-dimensional manifold $M^6=N^4 \times \mathbb{T}^2$ there
exist solutions to the SUSY equations IIA and IIB. More precisely,
if $\beta^1,\beta^2$ is a basis of 1-forms on the torus
$\mathbb{T}^2$, then $(\alpha, Re (\Omega)_{||}, Re
(\Omega)_{\perp}, Im (\Omega))$ given by
$$
Re(\alpha)= \beta^1,\ \ Im(\alpha)= \frac{1}{k_{||}} \beta^2,\ \ Re
(\Omega)_{||}= -k_{||}\, \omega_J,\ \ Re (\Omega)_{\perp}=
k_{\perp}\, \omega_K,\ \ Im (\Omega)= \omega_I,
$$
solves equations \eqref{SUSYeqIIA}, and $(\alpha, Re (\Omega)_{||},
Re (\Omega)_{\perp}, Im (\Omega))$ given by
$$
Re(\alpha)= \frac{1}{k_{||}} \beta^1,\ \ Im(\alpha)= \beta^2,\ \ Re
(\Omega)_{||}= k_{\perp}\, \omega_K,\ \ Re (\Omega)_{\perp}=
k_{||}\, \omega_I,\ \ Im (\Omega)= \omega_J,
$$
are solutions to equations \eqref{SUSYeqIIB}.
\end{proposition}

\begin{proof}
Let us consider $\hat \alpha_1=\beta^1+i\, \beta^2$, $\hat
\omega=\omega_K$ and $\hat \Omega=\omega_I+i\, \omega_J$. Since the
SU(3) structure $(\hat F=\beta^{12}+\omega_K,\ \hat \Psi=\hat
\alpha_1\wedge \hat \Omega)$ is integrable, the SU(2) structure
$(\hat J_1, \hat g_1, \hat \alpha_1, \hat \omega, \hat \Omega)$ is
obviously symplectic half-flat. Moreover, for any $k_{||}\in (0,1)$
the $SU(2)$ structure $(\hat J_2, \hat g_2, \hat \alpha_2= k_{||}\,
\beta^2 - i \frac{\beta^1}{k_{||}}, \hat \omega, \hat \Omega)$ is
also half-flat and then the result follows from Theorems~\ref{IIA}
and~\ref{IIB}.
\end{proof}

Notice that any compact hyperk\"ahler surface is either a
complex torus with a flat metric or a K3-surface with Calabi-Yau metric \cite{Besse}. Also observe that for the
solutions given in this proposition all the fluxes vanish.

Next we generalize the previous proposition by means of hypo
structures on 5-manifolds. We recall that an $SU(2)$ structure on a
5-manifold $P^5$ is an $SU(2)$-reduction of the principal bundle of
linear frames on $P$, equivalently a triple $(\eta,\omega_1,\Phi)$,
where $\eta$ is a $1$-form, $\omega_1$ is a $2$-form and $\Phi
=\omega_2+i\,\omega_3$ is a complex $2$-form on $P$ such that
$$
\eta\wedge\omega_1\wedge\omega_1 \neq 0\,, \quad\quad \Phi^2=0\,,
\quad\quad \omega_1\wedge\Phi =0\,, \quad\quad
\Phi\wedge\overline{\Phi} =2\,\omega_1\wedge\omega_1\,,
$$
and $\Phi$ is of type $(2,0)$ with respect to $\omega_1$.
Following~\cite{CS}, a $\SU(2)$ structure on a 5-manifold $P^5$ is said to be {\em hypo}
if $d\omega_1=d(\omega_2\wedge\eta)=d(\omega_3\wedge\eta)=0$.

\begin{proposition}\label{hypo}
Let $(P^5,\eta,\omega_1,\omega_2,\omega_3)$ be a compact
5-dimensional manifold endowed with a hypo SU(2) structure such that
$d\eta=0=d\omega_2$. Then, on the 6-dimensional manifold $M^6=P^5 \times
S^1$ there exist solutions to the SUSY equations IIA and IIB. More
precisely, if $\beta$ is a global nonvanishing 1-form on $S^1$ then
$(\alpha, Re (\Omega)_{||}, Re (\Omega)_{\perp}, Im (\Omega))$ given
by
$$
Re(\alpha)= \beta,\ \ Im(\alpha)= \frac{1}{k_{||}} \eta,\ \ Re
(\Omega)_{||}= -k_{||}\, \omega_3,\ \ Re (\Omega)_{\perp}=
k_{\perp}\, \omega_1,\ \ Im (\Omega)= \omega_2,
$$
solves equations \eqref{SUSYeqIIA}, and $(\alpha, Re (\Omega)_{||},
Re (\Omega)_{\perp}, Im (\Omega))$ given by
$$
Re(\alpha)= \frac{1}{k_{||}} \beta,\ \ Im(\alpha)= \eta,\ \ Re
(\Omega)_{||}= k_{\perp}\, \omega_3,\ \ Re (\Omega)_{\perp}=
k_{||}\, \omega_1,\ \ Im (\Omega)= \omega_2,
$$
are solutions to equations \eqref{SUSYeqIIB}.
\end{proposition}

\begin{proof}
It is clear that the SU(2) structure $(\hat J, \hat g, \hat
\alpha=\beta+i\, \eta, \hat \omega=\omega_1, \hat
\Omega=\omega_2+i\, \omega_3)$ on $M$ is symplectic half-flat. On
the other hand, the SU(2) structure $(\hat J_1, \hat g_1, \hat
\alpha_1=\beta+i\, \eta, \hat \omega=\omega_3, \hat
\Omega=\omega_1+i\, \omega_2)$ on $M$ is half-flat and, for any
$k_{||}\in (0,1)$, the $SU(2)$ structure $(\hat J_2, \hat g_2, \hat
\alpha_2= k_{||}\, \eta - i \frac{\beta}{k_{||}}, \hat \omega, \hat
\Omega)$ is also half-flat. Therefore, the result follows from
Theorems~\ref{IIA} and~\ref{IIB}.
\end{proof}

It is obvious that given a compact hyperk\"ahler 4-manifold $N^4$ we
can consider $P^5=N^4\times S^1$, but there are other manifolds to which
this result can be applied. For example, a nilmanifold corresponding
to the Lie algebra $(0,0,0,12,13)$ with the hypo structure
$\eta=e^1$, $\omega_1=e^{24}-e^{35}$, $\omega_2=e^{25}+e^{34}$ and
$\omega_3=e^{23}+e^{45}$. We will treat this example in more detail
in Section~\ref{Nilmanifolds}.

\section{New explicit solutions of the SUSY equations IIA and IIB}\label{new-examples}

In this section we show compact solvmanifolds admitting structures
solving the SUSY equations IIA and IIB. From Theorem~\ref{IIA} and
Proposition~\ref{balanced} we consider compact 6-solvmanifolds
admitting
both symplectic half-flat and Hermitian balanced $SU(3)$ structures.

\subsection{Nilmanifolds}\label{Nilmanifolds}

Conti and Tomassini classified \cite{CS} the nilmanifolds admitting
invariant symplectic half-flat  structures. It turns out that the
underlying nilpotent Lie algebra must be isomorphic to the abelian
Lie algebra, $\frh_6=(0,0,0,0,12,13)$ or $(0,0,0,12,13,23)$. Apart
from the abelian Lie algebra, only $\frh_6$ admits Hermitian balanced structure \cite{U}.
In fact, up to equivalence, there is a
1-parametric family of Hermitian balanced structures, which are
described as follows (see \cite{FIUV} for details). The complex
equations
$$
d\omega^1=d\omega^2=0, \quad d\omega^3 = \omega^{12} - \omega^{2\bar
1},
$$
define a complex structure $J$ on the Lie algebra $\frh_6$ and any
complex structure on $\frh_6$ is equivalent to $J$. With respect to $J$, any
Hermitian balanced structure
is equivalent to one and only one of the
form
$$
F_t=\frac{i}{2} (\omega^{1\bar{1}} + \omega^{2\bar{2}} +
t^2\,\omega^{3\bar{3}}),
$$
for some $t\not=0$.

Let us consider the basis of 1-forms $\{\beta^1,\ldots,\beta^6\}$
given by
$$
\beta^1+i\, \beta^4=\omega^1,\quad \beta^2+i\,
\beta^3=\omega^2,\quad \beta^5+i\, \beta^6= \frac{1}{2} \omega^3.
$$
In terms of this basis, we have the structure equations
\begin{equation}\label{ecus-h6}
  d \beta^1= d \beta^2= d \beta^3= d \beta^4 =0, \quad
  d \beta^5= \beta^{12},\quad
  d \beta^6= \beta^{13},
\end{equation}
and the complex structure $J$ and the fundamental form $F_t$ are
given by
\begin{equation}\label{JFt-h6}
J \beta^1=-\beta^4, J \beta^2=-\beta^3, J
\beta^5=-\beta^6,\quad\quad
F_t=\beta^{14}+\beta^{23}+4t^2\,\beta^{56}.
\end{equation}
Notice that the associated metric is $g_t=\beta^1\otimes
\beta^1+\cdots + \beta^4\otimes \beta^4+ 4t^2\, \beta^5\otimes
\beta^5+ 4t^2\,\beta^6\otimes \beta^6$. From now on we consider the
Hermitian balanced $SU(3)$ structure
$(F_t,\Psi_t)$ on $\frh_6$ given by
(\ref{JFt-h6}) and
\begin{equation}\label{balanced-h6}
\Psi_t = 2t\, (\beta^{1}+i\, \beta^{4})\wedge(\beta^{2}+i\,
\beta^{3})\wedge(\beta^{5}+i\, \beta^{6}).
\end{equation}

\medskip

\noindent $\bullet$ \emph{Solutions to equations IIB arising from
Hermitian balanced structures
on $\frh_6$:} For each $t\not=0$, the structure
$(F_t,\Psi_t)$ provides solutions to the SUSY equations IIB.
According to Theorem~\ref{IIB}, let us consider the half-flat SU(2)
structure $(\hat J_1, \hat g_1, \hat \alpha_1, \hat \omega,  \hat
\Omega)$ given by
$$\hat \alpha_1 = \beta^1+i\,
\beta^4,\quad \hat \omega=\beta^{23}+4t^2\,\beta^{56},\quad \hat
\Omega= 2t\, (\beta^{25}- \beta^{36}) +2 t\, i(\beta^{26}+
\beta^{35}).$$ By (\ref{ecus-h6}) the forms $\beta^{25}- \beta^{36}$
and $\beta^{26}+ \beta^{35}$ are closed, therefore for any
$k_{||}\in (0,1)$ we conclude that the $SU(2)$ structure $(\hat J_2,
\hat g_2, \hat \alpha_2= k_{||}\, Im(\hat \alpha_1)-i \frac{Re(\hat
\alpha_1)}{k_{||}}, \hat \omega, \hat \Omega)$ is half-flat.
Therefore, in terms of the basis $\{\beta^1,\ldots,\beta^6\}$ we get
the following explicit solutions $(\alpha, Re (\Omega)_{||}, Re
(\Omega)_{\perp}, Im (\Omega))$ of the SUSY equations
(\ref{SUSYeqIIB}):
$$
Re(\alpha)=\frac{1}{k_{||}} \beta^1,\quad\ Im(\alpha)=\beta^4,\quad\
Re (\Omega)_{||}=k_{\perp} (\beta^{23}+4t^2\,\beta^{56}),$$

$$Re (\Omega)_{\perp}=2t\, k_{||}(\beta^{25}- \beta^{36}),\quad\quad Im
(\Omega)= 2 t(\beta^{26}+ \beta^{35}),$$ where
$k_{\perp}=\sqrt{1-k_{||}^2}$. Notice that the fluxes are $H=0=F_1$,
and $e^{i\theta}g_s*F_3=4t^2(\beta^{126}-\beta^{135})$.

\medskip

\noindent $\bullet$ \emph{Solutions of the SUSY equations IIA on $\frh_6$:} For
each $t\not=0$, we consider the SU(2) structure $(\hat J, \hat  g,
\hat \alpha, \hat \omega, \hat \Omega)$ given by
\begin{equation}\label{symplectic-h6}
\hat \alpha = \beta^1+i\, \beta^4,\quad \hat
\omega=2t\,\beta^{25}-2t\,\beta^{36},\quad \hat \Omega=
(\beta^{2}+2ti\, \beta^{5})\wedge(-\beta^{3}+2ti\, \beta^{6}).
\end{equation}
Since the forms $\beta^{14}$ and $\beta^{25}- \beta^{36}$ are
closed, and
$$d(Re(\hat \alpha\wedge \hat \Omega))=\beta^1\wedge d(\beta^{23}+ 4t^2\beta^{56})
+2t \beta^4\wedge d(\beta^{26}+ \beta^{35})=0,$$ we have that the
SU(2) structure is symplectic half-flat for any $t\not=0$. According
to Theorem~\ref{IIA}, since $d(Re( \hat \alpha)) =d\beta^1=0$, the
forms $(\alpha, Re (\Omega)_{||} , Re (\Omega)_{\perp}, Im
(\Omega))$ given by
$$
Re(\alpha)= \beta^1,\quad\ Im(\alpha)=\frac{1}{k_{||}}
\beta^4,\quad\ Re (\Omega)_{||}= -2t
k_{||}(\beta^{26}+\beta^{35}),$$

$$Re (\Omega)_{\perp}=2t\, k_{\perp}(\beta^{25}- \beta^{36}),\quad\quad Im
(\Omega)= -\beta^{23}- 4 t^2 \beta^{56},$$ provide solutions to the
SUSY equations (\ref{SUSYeqIIA}). Notice that the fluxes are
$H=-4t^2 \frac{k_{\perp}}{k_{||}} (\beta^{126}-\beta^{135})$,
$F_0=0$, $g_s * F_2=-4
\frac{t^2}{k_{||}^2}(\beta^{1246}-\beta^{1345})$ and $F_4= 0$.

\bigskip

Next we show that solutions to equations IIA (resp. IIB) in general
are not stable by small deformations inside the class of half-flat
structures. For that, we first show explicitly that any
Hermitian balanced structure
$(F_t,\Psi_t)$ on $\frh_6$ given by
(\ref{JFt-h6})--(\ref{balanced-h6}) can be deformed into a
symplectic half-flat structure (\ref{symplectic-h6}) along a curve
of half-flat structures. For each $\vartheta\in \mathbb{R}$, let us
consider the SU(3) structure $(F_t^{\vartheta},\Psi_t^{\vartheta})$
given by
$$
F_t^{\vartheta}=\beta^{14}+\cos \vartheta\, \beta^{23}+ 2t\sin
\vartheta\, \beta^{25}- 2t\sin \vartheta\, \beta^{36}+4t^2\cos
\vartheta\, \beta^{56},$$ and
$$
\Psi_t^{\vartheta}=(\beta^{1}+i\, \beta^{4})\wedge(\beta^{2}+i\,
\cos \vartheta\, \beta^{3}+2t i \sin \vartheta\,
\beta^{5})\wedge(-\sin \vartheta\, \beta^{3}+2t\cos \vartheta\,
\beta^{5}+ 2t i \beta^{6}).
$$
A direct calculation shows that $Re (\Psi_t^{\vartheta})$ is closed
and $dF_t^{\vartheta}=4t^2\cos \vartheta (\beta^{126}-\beta^{135})$,
which implies that $F_t^{\vartheta}\wedge dF_t^{\vartheta}=0$.
Therefore, the structure is half-flat for any $\vartheta$, and
$(F_t^{0},\Psi_t^{0})$ is the
Hermitian balanced structure
given by (\ref{JFt-h6})--(\ref{balanced-h6}), and
$(F_t^{\frac{\pi}{2}},\Psi_t^{\frac{\pi}{2}})$ is the symplectic
structure (\ref{symplectic-h6}).

Since $F_t^{\vartheta}$ is symplectic if and only if $\cos
\vartheta=0$, by Theorem~\ref{IIA} we have that the half-flat
structure $(F_t^{\vartheta},\Psi_t^{\vartheta})$ does not solve
equations (\ref{SUSYeqIIA}) for $\vartheta \in (0,\frac{\pi}{2})$.

On the other hand, let us fix $\vartheta$ and consider the half-flat
structure $(F_t^{\vartheta},\Psi_t^{\vartheta})$. For any
$\lambda\in (0,1)$, a direct calculation shows that the structure
$(F=F_t^{\vartheta}, \Phi_{\lambda})$ given by
$$
\Phi_{\lambda}=(\lambda \beta^{4}-i\,
\frac{\beta^{1}}{\lambda})\wedge(\beta^{2}+i\, \cos \vartheta\,
\beta^{3}+2t i \sin \vartheta\, \beta^{5})\wedge(-\sin \vartheta\,
\beta^{3}+2t\cos \vartheta\, \beta^{5}+ 2t i \beta^{6}).
$$
is half-flat if and only if $\sin \vartheta=0$. From
Theorem~\ref{IIB} we conclude that the half-flat structure
$(F_t^{\vartheta},\Psi_t^{\vartheta})$ does not provide a solution
to equations (\ref{SUSYeqIIB}) for $\vartheta \in
(0,\frac{\pi}{2})$.

Therefore, we have proved the following result:

\begin{proposition}\label{non-stable}
The half-flat structure $(F_t^{\vartheta},\Psi_t^{\vartheta})$ does
not solve neither $(\ref{SUSYeqIIA})$ nor~$(\ref{SUSYeqIIB})$ for
any $\vartheta\in (0,\frac{\pi}{2})$. Therefore, solutions to the
SUSY equations IIA or IIB in general are not stable by small
deformations inside the class of half-flat structures.
\end{proposition}

\subsection{Compact solvmanifolds}

In this section we describe in detail two compact solvmanifolds
solving the SUSY equations IIA and IIB.

\subsubsection{Example}\label{solvable-1} Let us consider the $6$-dimensional $2$-step  completely solvable  Lie algebra
${\mathfrak s}_1=(0,0,13,-14,15,-16)$ with structure equations
\begin{equation}\label{ecus-solvable1}
d\beta^1=d\beta^2=0,\quad d\beta^3= \beta^{13},\quad
d\beta^4=-\beta^{14},\quad d\beta^5= \beta^{15},\quad
d\beta^6=-\beta^{16}.
\end{equation}
The corresponding simply-connected Lie group $S_1$ is isomorphic to
$\R \times  (\R \ltimes_{\phi} \R^4)$, where
$$
\phi (t) = \left (  \begin{array}{llll} e^t&0&0&0\\ 0&e^{-t}&0&0\\
0&0&e^t&0\\ 0&0&0&e^{-t} \end{array} \right ), \quad t \in \R.
$$

Since $\phi(1) = {\mbox {exp}}^{SL(4, \R)} (\phi' (0)) \in SL(4,
\Z)$, by \cite[Theorem 4]{Gor} we have that $\Gamma = \Z
\ltimes_{\phi} \Z^4$ is a lattice in $\R \ltimes_{\phi} \R^4$ and
therefore $\Z \times \Gamma = \Gamma_1$ is a lattice of $S_1$. By Hattori 's Theorem \cite{Hattori} we have that the de Rham cohomology of the compact quotient $S_1/\Gamma_1$ is isomorphic to the Chevalley-Eilenberg cohomology  $H^* ({\mathfrak s}_1)$ of ${\mathfrak s}_1$ and thus in particular $b_1 (S_1/\Gamma_1)=2$, $b_2 (S_1/\Gamma_1)=5$ and $b_3 (S_1/\Gamma_1)=8$.

Let us consider the almost complex structure
$$
J\beta^1=-\beta^2,\quad J\beta^3=-\beta^5,\quad J\beta^4=\beta^6.
$$
The basis of (1,0)-forms $\omega^1=\beta^1+i\,\beta^2$,
$\omega^2=\beta^3+i\,\beta^5$ and $\omega^3=-\beta^4+i\,\beta^6 $
satisfies
$$
d\omega^1=0,\quad d\omega^2=\frac12 \omega^{12}+\frac12
\omega^{\bar{1}2},\quad d\omega^3=-\frac12\omega^{13}-\frac12
\omega^{\bar{1}3}.
$$
Therefore, the almost complex structure $J$ is integrable. Since the
2-form $F=\beta^{12}+\beta^{35}-\beta^{46}$ satisfies that
$F^2=2(\beta^{1235}-\beta^{1246}-\beta^{3546})$ is closed, we get a
Hermitian balanced
SU(2) structure.

\medskip

\noindent $\bullet$ \emph{Solutions of equations IIB arising from the
Hermitian balanced structure
on~$\frs_1$:} The previous structure provides solutions to
the SUSY equations IIB. Let $(\hat J_1, \hat g_1, \hat \alpha_1,
\hat \omega, \hat \Omega)$ be the half-flat SU(2) structure given by
$$\hat \alpha_1 = \beta^1+i\, \beta^2,\quad \hat \omega=\beta^{35}-\beta^{46},\quad \hat \Omega=
-\beta^{34}- \beta^{56} + i(\beta^{36}+ \beta^{45}).$$ It follows
from (\ref{ecus-solvable1}) that the 2-forms $\beta^{34}$,
$\beta^{56}$ and $\beta^{36}+ \beta^{45}$ are closed, which implies
that for any $k_{||}\in (0,1)$ the $SU(2)$ structure $(\hat J_2,
\hat g_2, \hat \alpha_2= k_{||}\, Im(\hat \alpha_1)-i \frac{Re(\hat
\alpha_1)}{k_{||}}, \hat \omega, \hat \Omega)$ is half-flat. By
Theorem~\ref{IIB} we get the following solutions $(\alpha, Re
(\Omega)_{||}, Re (\Omega)_{\perp}, Im (\Omega))$ of the SUSY
equations (\ref{SUSYeqIIB}):
$$
Re(\alpha)=\frac{1}{k_{||}} \beta^1,\quad\ Im(\alpha)=\beta^2,\quad\
Re (\Omega)_{||}=k_{\perp} (\beta^{35}-\beta^{46}),$$

$$Re (\Omega)_{\perp}= -k_{||}(\beta^{34}+ \beta^{56}),\quad\quad Im
(\Omega)= \beta^{36}+ \beta^{45},$$ where
$k_{\perp}=\sqrt{1-k_{||}^2}$. Notice that the fluxes are $H=0=F_1$,
and $e^{i\theta}g_s*F_3=2\beta^1\wedge(\beta^{35}+\beta^{46})$.

\medskip

\noindent $\bullet$ \emph{Solutions of the SUSY equations IIA on $\frs_1$:} Let
us consider the SU(2) structure $(\hat J, \hat  g, \hat \alpha, \hat
\omega, \hat \Omega)$ given by
$$\hat \alpha = \beta^1+i\,\beta^2,\quad \hat \omega=
\beta^{34} +\beta^{56},\quad \hat \Omega= (\beta^{3}
+i\beta^{4})\wedge (\beta^{5} +i\beta^{6}).
$$
Since the forms $\beta^{12}$, $\beta^{34}$ and $\beta^{56}$ are
closed, and
$$d(Re(\hat \alpha\wedge \hat \Omega))=\beta^1\wedge d(\beta^{35}- \beta^{46})
- \beta^2\wedge d(\beta^{36}+ \beta^{45})=0,$$ we have that the
SU(2) structure is symplectic half-flat. By Theorem~\ref{IIA}, since
$d(Re( \hat \alpha)) = d\beta^1=0$, the forms $(\alpha, Re
(\Omega)_{||} , Re (\Omega)_{\perp}, Im (\Omega))$ given by
$$
Re(\alpha)= \beta^1,\quad\ Im(\alpha)=\frac{1}{k_{||}}
\beta^2,\quad\ Re (\Omega)_{||}= - k_{||}(\beta^{36}+\beta^{45}),$$

$$Re (\Omega)_{\perp}=k_{\perp}(\beta^{34}+ \beta^{56}),\quad\quad Im
(\Omega)= \beta^{35}- \beta^{46},$$ provide solutions to the SUSY
equations (\ref{SUSYeqIIA}). The fluxes are $H=2
\frac{k_{\perp}}{k_{||}} (\beta^{135}+\beta^{146})$, $F_0=0$, $g_s *
F_2=-\frac{2}{k_{||}^2}(\beta^{1235}+\beta^{1246})$ and $F_4= 0$.

\medskip

As in the previous example, the particular solutions on $\frs_1$ to
equations IIA and IIB given above are not stable by small
deformations inside the class of half-flat structures. For each
$\vartheta\in \mathbb{R}$, the SU(2) structure
$(F^{\vartheta},\Psi^{\vartheta})$ given by
$$
F^{\vartheta}=\beta^{12}+\cos \vartheta (\beta^{34}+ \beta^{56}) +
\sin \vartheta (\beta^{35}- \beta^{46})$$ and
$$
\Psi^{\vartheta}=(\beta^{1}+i\, \beta^{2})\wedge(\beta^{3}+i\, \cos
\vartheta\, \beta^{4}+ i \sin \vartheta\, \beta^{5})\wedge(-\sin
\vartheta\, \beta^{4}+\cos \vartheta\, \beta^{5}+ i \beta^{6})
$$
is half-flat, and for $\vartheta=0$ (resp.
$\vartheta=\frac{\pi}{2}$) we get the symplectic (resp.
Hermitian balanced)
half-flat structure given above. A direct calculation shows that the
half-flat structure $(F^{\vartheta},\Psi^{\vartheta})$ does not
solve neither $(\ref{SUSYeqIIA})$ nor~$(\ref{SUSYeqIIB})$ for any
$\vartheta\in (0,\frac{\pi}{2})$.

\subsubsection{Example}\label{solvable-2} Let us consider the solvable Lie
algebra ${\frs}_2=(0,0,-13-24,-14+23,15+26,16-25)$, that is, there
is a basis of 1-forms $\{\beta^1,\ldots,\beta^6\}$ satisfying
\begin{equation}\label{ecus-solvable2}
\left\{
  \begin{aligned}
  &d \beta^1= d \beta^2=0, \\
  &d \beta^3= -\beta^{13}-\beta^{24},\\
  &d \beta^4= -\beta^{14}+\beta^{23},\\
  &d \beta^5= \beta^{15}+\beta^{26},\\
  &d \beta^6= \beta^{16}-\beta^{25}.
  \end{aligned}
\right.
\end{equation}
The existence of a lattice $\Gamma_2$ of $S_2$ of the associated simply connected
solvable Lie group $S_2$  was  proved in \cite{Yamada} (see also \cite{deTom}). The  de Rham cohomology  of the compact quotient $S_2/ \Gamma_2$ (also known as Nakamura manifold)  is  not isomorphic  to $H^*({\frs}_2)$ (see \cite{deTom,Yamada} and more recently \cite{Guan,CF} for the cohomology of solvmanifolds). In particular $b_1(S_2/ \Gamma_2) = 2$, $b_2(S_2/ \Gamma_2) = 5$ and $b_3 (S_2/ \Gamma_2) = 8$.

Let us consider the almost complex structure
$$
J\beta^1=-\beta^2,\quad J\beta^3=-\beta^4,\quad J\beta^5=-\beta^6.
$$
The basis of (1,0)-forms $\omega^1=\beta^1+i\,\beta^2$,
$\omega^2=\beta^3+i\,\beta^4$ and $\omega^3=\beta^5+i\,\beta^6 $
satisfies
$$
d\omega^1=0,\quad d\omega^2=\omega^{2\bar{1}},\quad
d\omega^3=-\omega^{3\bar{1}},
$$
that is, $J$ is integrable.

For each $t\in \mathbb{R}-\{0\}$, the SU(3) structure $(F_t,\Psi_t)$
given by
$$
F_t=t^2 \beta^{12} + \beta^{34}+\beta^{56},\quad\quad  \Psi_t = t\,
(\beta^{1}+i\, \beta^{2})\wedge(\beta^{3}+i\,
\beta^{4})\wedge(\beta^{5}+i\, \beta^{6}),
$$
defines a 1-parametric family of (non-equivalent)
Hermitian balanced
$SU(3)$ structures on $\frs_2$ and thus a $1$-parametric family of
(non-equivalent)
Hermitian balanced
$SU(2)$ structures.

 Notice that the associated metric is
$g_t=t^2\,\beta^1\otimes \beta^1+t^2\,\beta^2\otimes \beta^2 +
\beta^3\otimes \beta^3+ \cdots+ \beta^6\otimes \beta^6$.

\medskip

\noindent $\bullet$ \emph{Solutions to equations IIB arising from
Hermitian balanced structures
on $\frs_2$:} For each $t\not=0$, the structure
$(F_t,\Psi_t)$ provides solutions to the SUSY equations IIB.
According to Theorem~\ref{IIB}, we consider the half-flat SU(2)
structure $(\hat J_1, \hat g_1, \hat \alpha_1, \hat \omega,  \hat
\Omega)$ given by
$$\hat \alpha_1 = t\, \beta^1+i\,
t\, \beta^2,\quad \hat \omega=\beta^{34}+\beta^{56},\quad \hat
\Omega=  \beta^{35}- \beta^{46} + i(\beta^{36}+ \beta^{45}).$$ By
(\ref{ecus-solvable2}) the forms $\beta^{35}- \beta^{46}$ and
$\beta^{36}+ \beta^{45}$ are closed, therefore for any $k_{||}\in
(0,1)$ we conclude that the $SU(2)$ structure $(\hat J_2, \hat g_2,
\hat \alpha_2= k_{||}\, Im(\hat \alpha_1)-i \frac{Re(\hat
\alpha_1)}{k_{||}}, \hat \omega, \hat \Omega)$ is half-flat.
Therefore, in terms of the basis $\{\beta^1,\ldots,\beta^6\}$ we get
the following explicit solutions $(\alpha, Re (\Omega)_{||}, Re
(\Omega)_{\perp}, Im (\Omega))$ of the SUSY equations
(\ref{SUSYeqIIB}):
$$
Re(\alpha)=\frac{t}{k_{||}} \beta^1,\quad\
Im(\alpha)=t\,\beta^2,\quad\ Re (\Omega)_{||}=k_{\perp}
(\beta^{34}+\beta^{56}),$$

$$Re (\Omega)_{\perp}= k_{||}(\beta^{35}- \beta^{46}),\quad\quad Im
(\Omega)= \beta^{36}+ \beta^{45},$$ where
$k_{\perp}=\sqrt{1-k_{||}^2}$. Notice that the fluxes are $H=0=F_1$,
and $e^{i\theta}g_s*F_3=-2\beta^1\wedge(\beta^{34}-\beta^{56})$.

\medskip

\noindent $\bullet$ \emph{Solutions of the SUSY equations IIA on $\frs_2$:} For
each $t\not=0$, we consider the SU(2) structure $(\hat J, \hat  g,
\hat \alpha, \hat \omega, \hat \Omega)$ given by
$$\hat \alpha = t\, \beta^1+i\, t\, \beta^2,\quad \hat \omega=-
\beta^{36} - \beta^{45},\quad \hat \Omega= (\beta^{6}
+i\beta^{3})\wedge (\beta^{5} +i\beta^{4}).
$$
Since the forms $\beta^{12}$ and $\beta^{36}+ \beta^{45}$ are
closed, and
$$\frac{1}{t}\, d(Re(\hat \alpha\wedge \hat \Omega))=\beta^1\wedge d(\beta^{34}+ \beta^{56})
+ \beta^2\wedge d(\beta^{35}- \beta^{46})=0,$$ we have that the
SU(2) structure is symplectic half-flat for any $t\not=0$. According
to Theorem~\ref{IIA}, since $d(Re( \hat \alpha)) = t\,d\beta^1=0$,
the forms $(\alpha, Re (\Omega)_{||} , Re (\Omega)_{\perp}, Im
(\Omega))$ given by
$$
Re(\alpha)= t\,\beta^1,\quad\ Im(\alpha)=\frac{t}{k_{||}}
\beta^2,\quad\ Re (\Omega)_{||}= - k_{||}(\beta^{35}-\beta^{46}),$$

$$Re (\Omega)_{\perp}=-k_{\perp}(\beta^{36}+ \beta^{45}),\quad\quad Im
(\Omega)= -\beta^{34}- \beta^{56},$$ provide solutions to the SUSY
equations (\ref{SUSYeqIIA}). Notice that the fluxes are $H=2
\frac{k_{\perp}}{k_{||}} (\beta^{134}-\beta^{156})$, $F_0=0$, $g_s *
F_2=-2\frac{t}{k_{||}^2}(\beta^{1234}-\beta^{1256})$ and $F_4= 0$.

\vskip1cm

\noindent {\bf Acknowledgments.} We would like to thank the referees for their valuable comments and remarks
that have improved the paper.
This work has been partially
supported through Project MICINN (Spain) MTM2008-06540-C02-02,
Project MIUR ``Riemannian Metrics and Differentiable Manifolds" and
by GNSAGA of INdAM.

\smallskip

{\small


\begin{thebibliography}{33}

\bibitem{And} D. Andriot, New supersymmetric flux vacua with intermediate $SU(2)$-structure, \emph{J. High Energy Phys.} {\bf 0808}, 096 (2008).

\bibitem{AGMP} D. Andriot, E. Goi, R. Minasian, M. Petrini,
Supersymmetry breaking branes on solvmanifolds and de Sitter vacua in string theory, arXiv:1003.3774v1 [hep-th].


\bibitem{Besse} A. Besse, Einstein Manifolds, Springer, Berlin, 1987.


\bibitem{CFI} P.G. C\'amara, A. Font, L.E. Ib\'a\~nez,
Fluxes, moduli fixing and MSSM-like vacua in a simple IIA orientifold, \emph{J. High Energy Phys.} {\bf 0509}, 013 (2005).


\bibitem{CG} G. Cavalcanti, M. Gualtieri, Generalized complex structures on nilmanifolds,
\emph{J. Symplectic Geom.} {\bf 2}  (2004), 393--410.

\bibitem{ChF} S. Chiossi, A. Fino,  Conformally parallel $G_2$ structures on a class of
solvmanifolds, \emph{Math. Z.}  {\bf 252(4)} (2006), 825--848.

\bibitem{ChSal}  S. Chiossi, S. Salamon, The intrinsic torsion of $SU(3)$ and $G_2$ structures. In Differential  Geometry, Valencia 2001, pages 115Ð133. World Scientific, 2002.

\bibitem{CSw} S. Chiossi, A. Swann, $G_2$-structures with torsion from half-integrable
nilmanifolds, \emph {J. Geom. Phys.} {\bf  54} (2005), 262--285.

\bibitem{CF} S. Console, A. Fino,  On the de Rham cohomology of solvmanifolds, arXiv:0912.2006,
to appear in \emph{Ann. Sc. Norm. Super. Pisa Cl. Sci. (5)}.

\bibitem{C} D. Conti, Half-flat nilmanifolds, arXiv:0903.1175, to appear in \emph{Math. Ann.}.

\bibitem{CT} D. Conti, A. Tomassini, Special symplectic six-manifolds,
\emph{Q. J. Math.} {\bf 58} (2007), 297--311.

\bibitem{CS} D. Conti, S. Salamon, Generalized Killing spinors in dimension 5,
\emph{Trans. Amer. Math. Soc.} {\bf 359} (2007), 5319--5343.

\bibitem{Dall'Agata} G. Dall'Agata, On supersymmetric solutions of type IIB supergravity with general fluxes,
\emph{Nucl.Phys. B} {\bf 695} (2004), 243--266.

\bibitem{deTom} P. de Bartolomeis, A. Tomassini, On solvable generalized Calabi-Yau manifolds,
\emph{Ann. Inst. Fourier} {\bf 56} (2006), 1281--1296.

\bibitem{FIUV} M. Fern\'{a}ndez, S. Ivanov, L. Ugarte, R. Villacampa, Non-K\"ahler
heterotic string compactifications with non-zero fluxes and constant
dilaton, \emph{Commun. Math. Phys.} {\bf 288} (2009), 677--697.

\bibitem{Gor} V.V. Gorbatsevich, Symplectic structures and cohomologies on some
solv-manifolds, \emph{Siberian Math. J.} {\bf 44} (2003), no. 2, 260--274.

\bibitem{GMPT2} M. Gra\~na, R. Minasian, M. Petrini, A. Tomasiello, Supersymmetric backgrounds from generalized Calabi-Yau manifolds,
\emph{J. High Energy Phys.} {\bf 0408}, 046 (2004).

\bibitem{GMPT3} M. Gra\~na, R. Minasian, M. Petrini, A. Tomasiello, Generalized structures of $N=1$ vacua,
\emph{J. High Energy Phys.} {\bf 0511}, 020 (2005).

\bibitem{GMPT} M. Gra\~na, R. Minasian, M. Petrini, A. Tomasiello, A scan for new $N=1$ vacua on twisted tori,
\emph{J. High Energy Phys.} {\bf 0705}, 031 (2007).

\bibitem{Gualtieri} M. Gualtieri, Generalized Complex Geometry, Oxford University Dphil thesis, matr.DG/0401221.

\bibitem{Guan} D. Guan, Modification and the cohomology groups of compact solvmanifolds,
\emph{Electron. Res. Announc. Amer. Math. Soc.}  {\bf 13} (2007), 74--81.

\bibitem{HT} N. Halmagyi, A. Tomasiello, Generalized K\"ahler potentials from supergravity,
\emph{Comm. Math. Phys.} {\bf 291} (2009), 1--30.


\bibitem{Hattori} A. Hattori, Spectral sequence in the de Rham cohomology of fibre bundles, \emph{J. Fac. Sci. Univ. Tokyo Sect. 1} {\bf 8} (1960), 289Ð331.

\bibitem{H} N.J. Hitchin, Stable forms and special metrics. In: Fern\'andez, M., Wolf J. (eds.),
Global differential geometry: the mathematical legacy of Alfred Gray
(Bilbao, 2000), Contemp. Math. {\bf 288}, Amer. Math. Soc.,
Providence, RI, 2001, 70--89.

\bibitem{Hitchin} N. Hitchin, Generalized Calabi-Yau manifolds, \emph{Q. J. Math.} {\bf 54} (2003), 281--308.

\bibitem{KT} P. Koerber, D. Tsimpis, Supersymmetric sources, integrability and generalized-structure compactifications,
\emph{J. High Energy Phys.} {\bf 0708}, 082 (2007).

\bibitem{MPZ} R. Minasian, M. Petrini, A. Zaffaroni, Gravity duals to deformed SYM theories and generalized complex geometry,
\emph{J. High Energy Phys.} {\bf 0612}, 055 (2006).

\bibitem{U} L. Ugarte, Hermitian structures on six dimensional nilmanifolds, \emph{Transform. Groups} {\bf 12} (2007), 175--202.

\bibitem{Yamada} T. Yamada, A pseudo-K\"ahler structure on a nontoral compact complex parallelizable solvmanifold,
\emph{Geom. Dedicata} {\bf 112} (2005), 115--122.

\end{thebibliography}
\end{document}